\documentclass[12pt]{amsart}
\usepackage{setspace}
\usepackage{hyperref}
\usepackage[alphabetic]{amsrefs}
\usepackage{mathtools}
\usepackage{enumerate}
\usepackage{amsmath, amsthm, amssymb, thmtools, etoolbox}
\usepackage{pgfplots}
\usepackage{amsaddr}
\usepackage{comment}

\usepackage{mathrsfs}
\usetikzlibrary{arrows}
\usepackage{tikz-cd}
\usetikzlibrary{calc}
\allowdisplaybreaks

\theoremstyle{plain}
\newtheorem{thm}{Theorem}[section]
\newtheorem{defn}[thm]{Definition}

\newtheorem{prop}[thm]{Proposition}

\numberwithin{equation}{section}

\newcommand{\mydprime}{^{\prime\prime}\mkern-1.2mu}

\newcommand{\midarrow}{\tikz \draw[-triangle 90] (0,0) -- +(.1,0);}

\begin{document}
\title{Quantum Symmetry in Multigraphs (part II)}
\begin{abstract}
		This article is a continuation of \cite{Goswami2023}. In this article, we give an explicit construction of a \textbf{non-Bichon} type co-action on a multigraph that is, it preserves quantum symmetry of $(V,E)$ in our sense but not always in Bichon's sense (definition 3.7 of \cite{Goswami2023}). This construction itself is motivated from automorphisms of quantum graphs.
	\end{abstract}
 
 \author{Debashish Goswami}
\thanks{It is to be noted that Debashish Goswami is partially supported by J.C. Bose national fellowship awarded by D.S.T., Government of India.}
	\address{Statistics and Mathematics Unit, Indian Statistical Institute\\ 203, B.T. Road, Kolkata 700108, India\\\textnormal{email: \texttt{debashish\_goswami@yahoo.co.in}}}

	\author{Sk Asfaq Hossain}
\address{School of Mathematical Sciences\\ National Institute of Science Education and Research\\ Bhubaneswar, Odisha, India\\\textnormal{email: \texttt{asfaq1994@gmail.com}}}
 \dedicatory{Dedicated to the memory of Prof. K.R.Parthasarathy}
	\maketitle
 
\section{Introduction}
As this article is a continuation of \cite{Goswami2023}, before going to the content of this article, we briefly recall what has been done in \cite{Goswami2023}. We have generalised the well known notions of quantum symmetry in single edged graphs described by Bichon and Banica (\cite{Bichon2003} and \cite{Banica2005})to the realm of multigraphs. We use the term ``single edged" instead of ``simple" because we want to incorporate the presence of loops, that is, an edge with a single endpoint vertex, into the picture of simple graphs. In our framework, a multigraph consists of a finite vertex set $V$, a finite edge set $E$, a source map $s:E\rightarrow V$ and a target map $t:E\rightarrow V$. Moreover, we also assume that there is no isolated vertex, that is, each vertex is either source or target of some edge. Initially, the process of generalisation has lead to two different categories of co-actions on $(V,E)$ namely $\mathcal{C}^{Ban}_{(V,E)}$ and $\mathcal{C}^{Bic}_{(V,E)}$ coming from Banica's and Bichon's two different notions of quantum symmetry in single edged graphs.  Though the universal object in $\mathcal{C}^{Bic}_{(V,E)}$ exists and can be regarded as a \textbf{quantum automorphism group}, the universal object in $\mathcal{C}^{Ban}_{(V,E)}$ is \textbf{too large} to be called an automorphism group. Therefore we have introduced a new category of co-actions $\mathcal{C}^{sym}_{(V,E)}$ which sits between the former two and is deemed to be the correct generalisation of Banica's notion of quantum symmetry in the context of multigraphs.\par 
This article is dedicated to constructing \textbf{``non-Bichon"} type quantum automorphisms of a multigraph. By \textbf{``non-Bichon"} type we mean that the quantum automorphism constructed in this article might not be a member of $\mathcal{C}^{Bic}_{(V,E)}$ but always a member of $\mathcal{C}^{sym}_{(V,E)}$. It is worth mentioning that the approach taken in this article is drastically different than in \cite{Goswami2023}. Here we start with a co-action on the vertex level and analyse the induced action on the edge space, whereas previously in \cite{Goswami2023} we started with an action on the edge space $L^2(E)$ and induced quantum permutations on sets of initial and final vertices of the multigraph. \par  
The construction done here is largely motivated from the context of quantum graphs. Quantum graphs are non commutative generalisations of finite graphs which primarily arose from the study of quantum communication channels (\cite{duan2012zero}) and there has been quite a lot of work done in this direction (see \cite{weaver2015quantum}, \cite{chirvasitu2022random}, \cite{brannan2020bigalois} and references therein). Though there are different ways of describing  automorphisms of quantum graphs (see for instance, \cite{daws2022quantum}), it can be simply be seen as a trace preserving co-action on the non-commutative vertex algebra commuting with the adjacency operator. In this article, instead of starting with just the vertex algebra, we start with a co-action on a ``fibered" vertex algebra and determine the edge preserving condition via a co-action induced on the operator algebra of its GNS space. Though our construction is fairly simple, we believe that these ideas pave the way to describe a ``multigraph" version of quantum graph and its quantum automorphisms. \par
Now we give a brief layout of this article. Along with the notations and conventions introduced in section 2 of \cite{Goswami2023}, we introduce some extra notations in section \ref{extra_not} related to discussions in this article. In section \ref{C_(V,E)}, we define a category $\mathcal{C}_{(V,E)}$ consisting of co-actions on uniform multigraphs and describe its universal object $Q_{(V,E)}$. The reason we are considering the uniform case separately is that these co-actions might not be faithful in the non-uniform case. This issue has been resolved in section \ref{non_uniform} by restricting the co-action to a suitable Woronowicz subalgebra of $Q_{(V,E)}$. In section \ref{Q_(V,E)-action}, we show that the action of the CQG $Q_{(V,E)}$ on a uniform multigraph $(V,E)$ is an action of \textbf{non-Bichon} type. 

\section{More Notations}\label{extra_not}
Let $(V,E)$ be a multigraph (\textbf{not necessarily uniform}) with source and target maps $s:E\rightarrow V$ and $t:E\rightarrow V$. We also assume that there is no isolated vertex. Let us define $N=sup_{i,j\in V}\{|E^i_j|\}$. Once we fix an edge-labeling of $(V,E)$ (see definition 3.23 of \cite{Goswami2023}), any edge $\tau\in E$ can be written as 
\begin{equation*}
	\tau=(i,j)r\quad\text{where}\quad 1\leq r\leq N,\quad s(\tau)=i\quad\text{and}\quad t(\tau)=j.
\end{equation*}  	
We will be considering CQG co-actions on the algebra $C(V)\otimes C(X_N)$ where $X_N=\{1,2,...,N\}$. Any such co-action $\alpha$ can be written as 
\begin{equation*}
	\alpha(\chi_i\otimes \chi_r)=\sum^N_{\substack{s=1\\k\in V}}\chi_{k}\otimes \chi_{s}\otimes q^{ks}_{ir}.
\end{equation*}
 It is clear that the co-representation matrix $(q^{ks}_{ir})_{ks,ir}$ satisfy \textbf{quantum permutation relations} (see \cite{Wang1998}). As $\alpha$ can be also seen as a bi-unitary co-representation on the Hilbert space $L^2(V)\otimes L^2(X_N)$, it induces a co-action $Ad_{\alpha}$ on the operator space $B(L^2(V))\otimes B(L^2(X_N))$ given by the following formula: 
 \begin{equation*}
     Ad_{\alpha}(T)=\alpha(T\otimes 1)\alpha^*\quad\text{for all}\quad T\in B(L^2(V))\otimes B(L^2(X_N)).
 \end{equation*}We will also write that
 \begin{equation*}
 	B(L^2(V))\otimes B(L^2(X_N))=linear\:span\:\{e_{ij}\otimes f_{rs}|i,j\in V;r,s=1,..,N\}
 \end{equation*}
 where $e_{ij}$ and $f_{rs}$ are elementary matrices in $B(L^2(V))$ and $B(L^2(X_N))$ defined as
 \begin{equation*}
 	e_{ij}(\chi_{j'})=\delta_{jj'}\chi_i \quad\text{and} \quad f_{rs}(\chi_{s'})=\delta_{ss'}\chi_{r}
 \end{equation*}
 for all $i,j,j'\in V$ and $r,s,s'\in \{1,..,N\}$.
 \section{The category $\mathcal{C}_{(V,E)}$ and its universal object $Q_{(V,E)}$}\label{C_(V,E)}
 In this section and the next, we work with a fixed uniform multigraph $(V,E)$. We also fix an edge-labeling of $(V,E)$.
\begin{defn}\label{main_def_exp}
	The category of co-actions $\mathcal{C}_{(V,E)}$ consists of objects $(\mathcal{A},\Delta, \alpha)$ where $\alpha:C(V)\otimes C(X_N)\rightarrow C(V)\otimes C(X_N) \otimes \mathcal{A}$ is a co-action of a CQG $(\mathcal{A},\Delta)$ on the algebra $C(V)\otimes C(X_N)$ satisfying the following conditions:
	\begin{enumerate}
	\item $\alpha(C(V)\otimes 1)\subseteq C(V) \otimes1\otimes  \mathcal{A}$.
	\item $Ad_{\alpha}(\mathcal{\mathcal{E}})\subseteq \mathcal{E}\otimes \mathcal{A}$
	\end{enumerate}
where $\mathcal{E}\subseteq B(L^2(V))\otimes B(L^2(X_N))$ is a linear subspace defined by
\begin{equation*}
\mathcal{E}=linear\:span\:\{e_{ij}\otimes f_{rr}|\:E^i_j\neq\phi;1\leq r\leq |E^i_j|\}.
\end{equation*}
Morphisms in $\mathcal{C}_{(V,E)}$ are \textbf{quantum group homomorphisms} intertwining similar type co-actions.
\end{defn}
 Identifying $\mathcal{E}$ with $L^2(E)$ as vector spaces and $C(V)\otimes 1$ with $C(V)$ as algebras, condition (2) of definition \ref{main_def_exp} can be interpreted as $\alpha$ preserving the set of edges in $(V,E)$.
\begin{prop}\label{induced_perm_exp}
For $(\mathcal{A},\Delta, \alpha)\in \mathcal{C}_{(V,E)}$, $\alpha$ induces a \textbf{quantum permutation} $\alpha_V$ of the vertex set $V$ which is given by,
\begin{equation*}
	\alpha_V(\chi_i)=\sum_{k\in V} \chi_{k}\otimes Q^k_i
\end{equation*}
where $Q^k_i=\sum^N_{r=1}q^{ks}_{ir}$ for all $i,k\in V$ and $s\in\{1,..,N\}$. 
\end{prop}
\begin{proof}
	For $i\in V$, we have
	\begin{equation*}
		\alpha(\chi_i\otimes 1)=\sum^N_{\substack{s=1\\k\in V}}\chi_k \otimes \chi_{s}\otimes  (\sum^N_{r=1}q^{ks}_{ir}) 
	\end{equation*}
	From (1) of definition \ref{main_def_exp}, it follows that,
	\begin{equation}\label{vertex_algebra_prsrv}
		\sum^N_{r=1}q^{ks}_{ir}=\sum^N_{r=1}q^{ks'}_{ir}\quad\text{for all}\quad s,s'\in\{1,..,N\}.
	\end{equation}
	Therefore, the quantities $\{Q^k_i\:|\:k,i\in V\}$ in proposition \ref{induced_perm_exp} are well-defined. As $(q^{ks}_{ir})_{(ks),(ir)}$ is a \textbf{quantum permutation matrix}, it follows that, the co-representation matrix $(Q^k_i)_{k,i\in V}$ is also \textbf{quantum permutation matrix} making $\alpha_V$ a co-action on $C(V)$.
\end{proof}
\begin{prop}\label{edge_presrv_exp}
	Let $(\mathcal{A},\Delta,\alpha)\in\mathcal{C}_{(V,E)}$ and $(q^{ks}_{ir})_{(ks),(ir)}$ be the co-representation matrix of $\alpha$. For $(i,j)r\in E; s,s'\in \{1,..,N\}$ and $k,l\in V$ the following conditions are true:
	\begin{enumerate}
		\item $q^{ks}_{ir} q^{ls'}_{jr}=0$ if $E^k_l=\phi$.
		\item $q^{ks}_{ir}q^{ls'}_{jr}=0$  if $E^k_l\neq\phi$ and $s\neq s'$.
	\end{enumerate}
\end{prop}
\begin{proof}
For $i,j,i'\in V$ and $r,s,r'\in\{1,..,N\}$ we observe that,
\begin{align*}
Ad_{\alpha}(e_{ij}\otimes f_{rs})(\chi_{i'}\otimes \chi_{r'}\otimes 1)&=\alpha(e_{ij}\otimes f_{rs})(\sum^N_{\substack{r\mydprime=1\\i\mydprime\in V}}\chi_{i\mydprime}\otimes \chi_{r\mydprime}\otimes q_{i\mydprime r\mydprime}^{i'r'})\\
&=\alpha(\chi_{i}\otimes \chi_{r}\otimes q_{js}^{i'r'})\\
&=\sum^N_{\substack{r\mydprime=1\\i\mydprime\in V}}\chi_{i\mydprime}\otimes \chi_{r\mydprime}\otimes q^{i\mydprime r\mydprime}_{ir}q_{js}^{i'r'}.
\end{align*}
Therefore the map $Ad_{\alpha}:B(L^2(V))\otimes B(L^2(X_N))\rightarrow B(L^2(V))\otimes B(L^2(X_N))\otimes \mathcal{A}$ is given by,
\begin{equation}\label{Ad_alpha_expansion}
Ad_{\alpha}(e_{ij}\otimes f_{rs})=\sum^N_{\substack{r\mydprime,s'=1\\i\mydprime,j'\in V}} e_{i\mydprime j'}\otimes f_{r\mydprime s'}\otimes q^{i\mydprime r\mydprime}_{ir}q^{j' s'}_{js}.
\end{equation}
From condition (2) of definition \ref{main_def_exp} and equation \ref{Ad_alpha_expansion}, the proposition follows.
\end{proof}
It should be noted that the identities in equation \ref{vertex_algebra_prsrv} and the ones in proposition \ref{edge_presrv_exp} are complete set of algebraic relations satisfied by the coefficients of a co-representation matrix. In light of this, using standard techniques from theory of compact quantum groups (see for example, theorem 4.4.3 of \cite{hossain2023quantum}), we can prove the following result:
\begin{thm}
The category $\mathcal{C}_{(V,E)}$ admits a universal object namely $(Q_{(V,E)},\Delta,\alpha)$. The algebra $Q_{(V,E)}$ is the \textbf{universal C* algebra} generated by the set of elements $\{q^{ks}_{ir}\:|\:k,i\in V;s,r=1,..,N\}$ satisfying the following conditions:
\begin{enumerate}
\item Coefficients of the matrix $(q^{ks}_{ir})_{(ks),(ir)}$ satisfy \textbf{quantum permutation relations}.
\item $q^{ks}_{ir} q^{ls'}_{jr}=0$ if $E^k_l=\phi$ and $(i,j)r\in E$.
\item $q^{ks}_{ir}q^{ls'}_{jr}=0$  if $E^k_l\neq\phi$, $s\neq s'$ and $(i,j)r\in E$.
\end{enumerate} 
Moreover, the co-product $\Delta$ on $Q_{(V,E)}$ is given by $$\Delta(q^{ks}_{ir})=\sum^N_{\substack{s'=1\\k'\in V}}q^{ks}_{k's'}\otimes q^{k's'}_{ir}\quad\text{where}\quad k,i\in V; s,r=1,..,N.$$
The canonical co-action of $Q_{(V,E)}$ on $C(V)\otimes C(X_N)$ is given by,
\begin{equation*}
	\alpha(\chi_i\otimes \chi_r)=\sum^N_{\substack{s=1\\k\in V}}\chi_k\otimes \chi_s\otimes q^{ks}_{ir}\quad\text{for all}\quad i\in V; r=1,..,N.
\end{equation*}
\end{thm}
\section{Action of $Q_{(V,E)}$ on $(V,E)$}\label{Q_(V,E)-action}

 In this section, we see that the CQG $Q_{(V,E)}$ defined in previously has a \textbf{non-Bichon} type co-action on the uniform multigraph $(V,E)$ \textbf{preserving its quantum symmetry in our sense}.
\begin{prop}\label{bimod_exp}
	There exists a bi-unitary co-representation $\beta:L^2(E)\rightarrow L^2(E)\otimes Q_{(V,E)}$  of the CQG $(Q_{(V,E)},\Delta)$ such that the pair $(\alpha_V,\beta)$ respects $C(V)-C(V)$ bimodule structure on $L^2(E)$, that is, for all $i\in V$ and $\sigma\in E$,
	\begin{equation}
		\beta(\chi_i.\chi_{\sigma})=\alpha_V(\chi_i)\beta(\chi_\sigma)\quad\text{and}\quad\beta(\chi_\sigma.\chi_i)=\beta(\chi_\sigma)\alpha_V(\chi_i).
	\end{equation} 
	For a member of $\mathcal{C}_{(V,E)}$, the co-action $\alpha_V:C(V)\rightarrow C(V)\otimes Q_{(V,E)}$ is described in proposition \ref{induced_perm_exp}.
 \end{prop}
\begin{proof}
We start with identifying $L^2(E)$ with $\mathcal{E}$ (see definition \ref{main_def_exp}) as vector spaces through the following identification of basis elements:
\begin{equation*}
	\chi_{(i,j)r} \longleftrightarrow e_{ij}\otimes f_{rr} \quad \text{for all}\quad (i,j)r\in E.
\end{equation*}
As $Ad_{\alpha}$ on $B(L^2(V))\otimes B(L^2(X_N))$  preserves $\mathcal{E}$ by definition \ref{main_def_exp}, using the identification mentioned above it induces a map $\beta:L^2(E)\rightarrow L^2(E)\otimes\mathcal{A}$ which can be written as (using equation \ref{Ad_alpha_expansion}),
\begin{equation*}
	\beta(\chi_{(i,j)r})=\sum_{(k,l)s\in E}\chi_{(k,l)s}\otimes q^{ks}_{ir}q^{ls}_{jr}\quad\text{for all} \quad (i,j)r\in E.
\end{equation*}  
By taking $u^{(k,l)s}_{(i,j)r}=q^{ks}_{ir}q^{ls}_{jr}$ for $(i,j)r, (k,l)s\in E$, it follows that the matrix $U=(u^{\sigma}_{\tau})_{\sigma,\tau\in E}$ is the co-representation matrix of $\beta$. To conclude that $\beta$ is a bi-unitary co-representation, it is enough to show that the matrix $U$ is a bi-unitary matrix and satisfies matrix co-product identity, that is, $\Delta(u^{\sigma}_{\tau})=\sum_{\tau'\in E}u^{\sigma}_{\tau'}\otimes u^{\tau'}_{\tau}$ for all $\sigma,\tau\in E$. For $(k,l)s, (k',l')s'\in E$, using proposition \ref{edge_presrv_exp}, we observe the following:
\begin{align*}
\sum_{(i,j)r\in E}u^{(k,l)s}_{(i,j)r}u^{(k',l')s'*}_{(i,j)r}&=\delta_{l,l'}\delta_{s,s'}\sum_{(i,j)r\in E}q^{ks}_{ir}q^{ls}_{jr}q^{k's}_{ir}=\delta_{l,l'}\delta_{s,s'}\sum^N_{\substack{r,r'=1\\i,j\in V}}q^{ks}_{ir}q^{ls}_{jr'}q^{k's}_{ir}=\delta_{k,k'}\delta_{l,l'}\delta_{s,s'},\\
\sum_{(i,j)r\in E}u^{(i,j)r*}_{(k,l)s}u^{(i,j)r}_{(k',l')s'}&=\delta_{k,k'}\delta_{s,s'}\sum_{(i,j)r\in E}q^{jr}_{ls}q^{ir}_{k's}q^{jr}_{l's}=\delta_{k,k'}\delta_{s,s'}\sum^N_{\substack{r,r'=1\\i,j\in V}}q^{jr}_{ls}q^{ir'}_{k's}q^{jr}_{l's}=\delta_{k,k'}\delta_{l,l'}\delta_{s,s'}.
\end{align*}
Similarly, it also follows that the contragradient co-representation $\overline{\beta}$ (with co-representation matrix $(u^{\sigma*}_\tau)_{\sigma,\tau\in E}$) is also unitary. The matrix co-product identity holds:
\begin{align*}
\Delta(u^{(i,j)r}_{(k,l)s})=\Delta(q^{ir}_{ks}q^{jr}_{ls})&=\sum^N_{\substack{s',s\mydprime=1\\k'l'\in V}}q^{ir}_{k's'}q^{jr}_{l's\mydprime}\otimes q^{k's'}_{ks}q^{l's\mydprime}_{ls}\\
&=\sum_{(k'l')s'\in E} q^{ir}_{k's'}q^{jr}_{l's'}\otimes q^{k's'}_{ks}q^{l's'}_{ls}\\
&=\sum_{(k',l')s'\in E}u^{(i,j)r}_{(k'l')s'}\otimes u^{(k',l')s'}_{(k,l)s}.
\end{align*}
For $k',l'\in V$ and $(i,j)r, (k,l)s\in E$, we observe that,
\begin{align*}
Q^{i}_{k'}u^{(i,j)r}_{(k,l)s}&=\sum^N_{s'=1}q^{ir}_{k's'}q^{ir}_{ks}q^{jr}_{ls}=\delta_{k,k'}q^{ir}_{ks}q^{jr}_{ls}=\delta_{k,k'}u^{(i,j)r}_{(k,l)s},\\
u^{(i,j)r}_{(k,l)s}Q^{j}_{l'}&=\sum^N_{s'=1}q^{ir}_{ks}q^{jr}_{ls}q^{js}_{l's'}=\delta_{l,l'}q^{ir}_{ks}q^{jr}_{ls}=\delta_{l,l'}u^{(i,j)r}_{(k,l)s}.
\end{align*}
The observations made above indicates that the pair $(\alpha_V,\beta)$ preserves $C(V)-L^2(E)-C(V)$ bimodule structure. Therefore proposition \ref{edge_presrv_exp} is proved.
\end{proof}
From \cite{Goswami2023} we have the following: Any bi-unitary co-representation $\beta:L^2(E)\rightarrow L^2(E)\otimes \mathcal{A}$ of a CQG $(\mathcal{A},\Delta)$ preserves \textbf{quantum symmetry of $(V,E)$ in Banica's sense} if the following holds:
\begin{enumerate}
    \item There exists a co-action $\alpha:C(V)\rightarrow C(V)\otimes \mathcal{A}$ such that the pair $(\alpha,\beta)$ respects $C(V)-L^2(E)-C(V)$ bimodule structure.
    \item $\beta(\xi)=\xi\otimes 1$ where $\xi=\sum_{i\in V}\chi_i$.
\end{enumerate}
In our case, the bi-unitary map $\beta$ in proposition \ref{bimod_exp} is a co-action of $Q_{(V,E)}$ on the multigraph $(V,E)$ \textbf{preserving its quantum symmetry in Banica's sense}. The next proposition proves that it is in fact a member of $\mathcal{C}^{sym}_{(V,E)}$, that is, it satisfies \textbf{restricted orthogonality} (see definition 3.17 of \cite{Goswami2023}).
\begin{prop}
	The bi-unitary co-representation $\beta$ of $Q_{(V,E)}$ on $L^2(E)$ mentioned in proposition \ref{bimod_exp} preserves \textbf{quantum symmetry of $(V,E)$ in our sense.} 
\end{prop}
\begin{proof}
The map $\beta:L^2(E)\rightarrow L^2(E)\otimes Q_{(V,E)}$ is given by
\begin{equation*}
\beta(\chi_{(i,j)r})=\sum_{(k,l)s\in E}\chi_{(k,l)s}\otimes q^{ks}_{ir}q^{ls}_{jr}\quad\text{where}\quad (i,j)r\in E.
\end{equation*}
To show that the coefficients of co-representation matrix of $\beta$ satisfies \textbf{restricted orthogonality} (see proposition 3.18 of \cite{Goswami2023}), we observe the following: For $(i,j)r$, $(i,j)r', (k,l)s\in E$ and $r\neq r'$,
\begin{align*}
	u^{(i,j)r}_{(k,l)s}u^{(i,j)r'*}_{(k,l)s}&=q^{ir}_{ks}(q^{jr}_{ls}q^{jr'}_{ls})q^{ir'}_{ks}=0,\\
	u^{(i,j)r*}_{(k,l)s}u^{(i,j)r'}_{(k,l)s}&=q^{jr}_{ls}(q^{ir}_{ks}q^{ir'}_{ks})q^{jr'}_{ls}=0.
\end{align*} 
Therefore $\beta$ is a co-action of $Q_{(V,E)}$ on $(V,E)$ preserving its \textbf{quantum symmetry in our sense}. 
\end{proof}
\begin{prop}\label{non_Bichon_type_action}
	The co-action $\beta$ of $Q_{(V,E)}$ on the multigraph $(V,E)$ is an action of \textbf{non-Bichon} type, that is, the triplet $(Q_{(V,E)},\Delta,\beta) \notin \mathcal{C}^{Bic}_{(V,E)}$ whenever $\mathcal{C}^{Bic}_{(V,E)}\neq \mathcal{C}^{sym}_{(V,E)}$.
\end{prop}
\begin{proof}
As $Q_{(V,E)}$ is universal in $\mathcal{C}_{(V,E)}$, it is enough to show that there exists a member of $\mathcal{C}_{(V,E)}$ which is of \textbf{non Bichon} type, that is, not a member of $\mathcal{C}^{Bic}_{(V,E)}$ whenever $\mathcal{C}^{Bic}_{(V,E)}\neq \mathcal{C}^{sym}_{(V,E)}$. From corollary 3.21 of \cite{Goswami2023} it follows that
\begin{equation*}
\mathcal{C}^{Bic}_{(V,E)}\neq \mathcal{C}^{sym}_{(V,E)}\quad\iff\quad \mathcal{C}^{Bic}_{(V,\overline{E})}\neq \mathcal{C}^{sym}_{(V,\overline{E})}
\end{equation*}
where $(V,\overline{E})$ is the underlying single edged graph of $(V,E)$. The underlying single edged graph of $(V,E)$ is a single edged graph with same vertex set $V$ and edge set $\overline{E}\subseteq V\times V $ which is given by
\begin{equation*}
    (i,j)\in \overline{E}\quad\text{if and only if}\quad E^i_j\neq \phi.
\end{equation*}
Let $(Q,\Delta)$ be the quantum automorphism group of $(V,\overline{E})$ preserving its quantum symmetry in Banica's sense. We assume $U=(u^i_j)_{i,j\in V}$ to be the matrix of canonical generators of $Q$ satisfying \textbf{quantum permutation relations} and $UA=AU$ where $A$ is the adjacency matrix of $(V,E)$. As $U$ commutes with $A$, from theorem 2.2 of \cite{Banica2005} we have $$u^k_iu^l_j=0 \quad\text{if}\quad |E^k_l|\neq |E^i_j|.$$
 Let us define a map $\alpha:C(V)\otimes C(X_N)\rightarrow C(V)\otimes C(X_N)\otimes Q$ by the following formula:
 \begin{equation}\label{Q^ban_action_exp}
 	\alpha(\chi_i\otimes \chi_r)=\sum^N_{\substack{s=1\\k\in V}}\chi_{k}\otimes \chi_s\otimes \delta_{s,r}u^{k}_{i}
\end{equation}
It is easy to see that $\alpha$ is a co-action of $(Q,\Delta)$ on the algebra $C(V)\otimes C(X_N)$ satisfies (1) of definition \ref{main_def_exp}. From equation \ref{Q^ban_action_exp}, it further follows that the coefficients of the co-representation matrix of $\alpha$ satisfy identities in proposition \ref{edge_presrv_exp}. Therefore we have  $(Q,\Delta,\alpha)\in \mathcal{C}_{(V,E)}$. Moreover, the action of $(Q,\Delta)$ on the edge space $L^2(E)$ is given by,   
\begin{equation*}
	\beta(\chi_{(i,j)r})=\sum_{(k,l)s\in E}\chi_{(k,l)s}\otimes \delta_{s,r}u^k_iu^l_j\quad\text{for all}\quad (i,j)r\in E.
\end{equation*}
 This action $\beta$ is of \textbf{non-Bichon} type because $\beta$ is not necessarily a \textbf{quantum permutation} on the edge set. The fact that $\beta$ is not necessarily a \textbf{quantum permutation} follows from the condition that  $(Q,\Delta, U^{(2)})\notin \mathcal{C}^{Bic}_{(V,\overline{E})}$ whenever  $\mathcal{C}^{Bic}_{(V,\overline{E})}\neq \mathcal{C}^{sym}_{(V,\overline{E})}$.  
\end{proof}
\section{The non-uniform case}\label{non_uniform}
Let $(V,E)$ be a multigraph which is \textbf{not necessarily uniform}. After fixing an  edge-labeling of $(V,E)$ and considering $N=sup_{i,j\in v}\{|E^i_j|\}$, the category $\mathcal{C}_{(V,E)}$  defined in definition \ref{main_def_exp} makes sense. The universal object in $\mathcal{C}_{(V,E)}$, namely $Q_{(V,E)}$,  exists and is the \textbf{universal C* algebra} generated by coefficients of the matrix $\{q^{ks}_{ir}\:|\:k,i\in V;s,r=1,..,N\}$ satisfying the following relations:
\begin{enumerate}
\item The coefficients of the matrix $(q^{ks}_{ir})_{(ks),(ir)}$ satisfy \textbf{quantum permutation relations}.
\item $q^{ks}_{ir} q^{ls'}_{jr}=0$ if $E^k_l=\phi$ and $(i,j)r\in E$.
\item $q^{ks}_{ir}q^{ls'}_{jr}=0$  if $E^k_l\neq\phi$, $s\neq s'$ and $(i,j)r\in E$
\item  $q^{ks}_{ir}q^{ls}_{jr}=0$ if $E^k_l\neq\phi$, $s>|E^k_l|$ and $(i,j)r\in E$.
\end{enumerate}
The co-product $\Delta$ on $Q_{(V,E)}$ is the matrix co-product identical to the uniform case. 
\subsection*{A word of caution:} In case of a non-uniform multigraph $(V,E)$, action of the CQG $Q_{(V,E)}$ on the edge space $L^2(E)$ might \textbf{NOT} be faithful even in a classical scenario as we see for the multigraph in figure \ref{disjoint_union} where the vertex set is given by $\{a,b,c,d\}$, and the edge set consists of $7$ elements with $5$ edges from $a$ to $b$ and $2$ edges from $c$ to $d$.  It can be seen that keeping all the vertices and edges fixed, the absent (dashed) edges from $c$ to $d$ can be freely permuted producing identity isomorphism of the multigraph each time. To fix this issue, our initial action can be restricted to the ``permissible" pairs as we will see below.
\begin{figure}[h]
	\centering
	\begin{tikzpicture}[scale=4]
		\node at (-0.5,0) {$\bullet$};
		\node at (-1.5,0) {$\bullet$};
		\node at (0.5,0) {$\bullet$};
		\node at (1.5,0) {$\bullet$};
        \draw[thick, out=60, in=120, looseness=1.2] (0.5,0) to node{\midarrow} (1.5,0); 
        \draw[dashed, out=-60, in=-120, looseness=1.2] (0.5,0) to node{\midarrow} (1.5,0); 
        \draw[thick, out=60, in=120, looseness=1.2] (-1.5,0) to node{\midarrow} (-0.5,0); 
        \draw[thick, out=-60, in=-120, looseness=1.2] (-1.5,0) to node{\midarrow} (-0.5,0); 
		\draw[thick] (-1.5,0) edge[bend left] node{\midarrow} (-0.5,0);
		\draw[thick] (-1.5,0) edge node{\midarrow} (-0.5,0);
		\draw[thick] (-1.5,0) edge[bend right] node{\midarrow} (-0.5,0);
		\draw[dashed] (0.5,0) edge[bend right] node{\midarrow} (1.5,0);
		\draw[dashed] (0.5,0) edge node{\midarrow} (1.5,0);
		\draw[thick] (0.5,0) edge[bend left] node{\midarrow} (1.5,0);
		\node at (-1.5,-.2) {$a$};
		\node at (-0.5,-0.2) {$b$};
		\node at (0.5,-0.2) {$c$};
		\node at (1.5, -0.2) {$d$};
	\end{tikzpicture}
	\caption{A non-uniform multigraph}
	\label{disjoint_union} 
\end{figure}
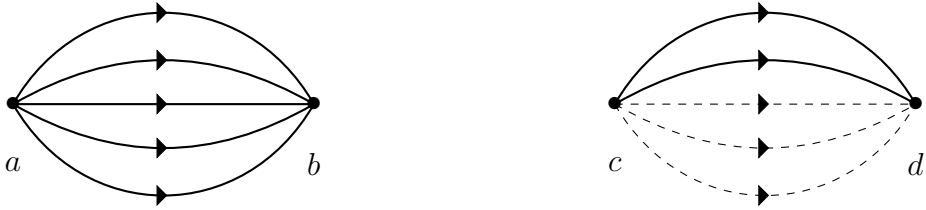

\begin{defn}
	Given a multigraph $(V,E)$ and its fixed representation,
a pair $(k,s)\in V\times X_N$ is called \textbf{permissible} if there exists $\tau\in E$ such that $\tau=(k,l)s$ or $\tau=(l,k)s$ for some $l\in V$.	
\end{defn}
\begin{prop}
Let $\alpha$ be the canonical co-action of $Q_{(V,E)}$ on $C(V)\otimes C(X_N)$. Then it follows that 
\begin{equation*}
	\alpha(\mathcal{P})\subseteq \mathcal{P}\otimes \mathcal{A}\quad\text{where}\quad \mathcal{P}=\text{linear span}\:\{\chi_k\otimes \chi_s\:|\:(k,s)\:\text{is a permissible pair}\}.
\end{equation*}
\end{prop}
\begin{proof}
Let $i,k\in V$ and $r,s\in \{1,..,N\}$ be such that $(i,r)$ is \textbf{permissible} but $(k,s)$ is not. Without loss of generality, we assume that there exists $j\in V$ such that $\tau=(i,j)r$.
To prove the proposition, it is enough to show that $q^{ks}_{ir}=0$.
\begin{equation}\label{permissible_pair_prsv}
	q^{ks}_{ir}=\sum^N_{\substack{s'=1\\l\in V}}q^{ks}_{ir}q^{ls'}_{jr}=0.
\end{equation}
\end{proof}
To produce a quantum automorphism group of $(V,E)$, we need to consider the Woronowicz C* subalgebra generated by elements corresponding to \textbf{permissible} pairs. More precisely, we have the following proposition:
\begin{prop}
Let $Q'_{(V,E)}$ be the C* subalgebra of $Q_{(V,E)}$ given by,
\begin{align*}
	Q'_{(V,E)}&=C^*\{q^{ks}_{ir}\:|\:\text{$(k,s)$ and $(i,r)$ are ``permissible" pairs}\}\\
	&=C^*\{q^{ks}_{ir}q^{ls}_{jr}\:|\:(k,l)s, (i,j)r\in E\}. 
\end{align*}
The co-product $\Delta$ on $Q_{(V,E)}$ restricts to the subalgebra $Q'_{(V,E)}$ making it a Woronowicz C* subalgebra of $Q_{(V,E)}$. Moreover, there exists a faithful \textbf{non-Bichon} type co-action  $\beta$ of $Q'_{(V,E)}$ on $(V,E)$ preserving its \textbf{quantum symmetry in our sense.}
\end{prop}
\begin{proof}
The fact that the co-product $\Delta$ of $Q_{(V,E)}$ restricts to $Q'_{(V,E)}$ follows from equation \ref{permissible_pair_prsv}. It is easy to see that induced action of $Q'_{(V,E)}$ on the edge space $L^2(E)$ is faithful. To prove that this action is an action of \textbf{non-Bichon} type, we use identical arguments used in proof of proposition \ref{non_Bichon_type_action} for the \textbf{underlying weighted single edged graph} $(V,\overline{E},w)$ where $w:\overline{E}\rightarrow \mathbb{C}$ is a weight function on the set of edges $\overline{E}$ given by
\begin{equation*}
    w((i,j))=|E^i_j|\quad \text{for all}\quad (i,j)\in \overline{E}.
\end{equation*}
\end{proof}

\bibliographystyle{alphaurl}
\bibliography{quantum_symmetry_in_multigraphs}
\end{document}